\documentclass[12pt,a4paper]{article}
\usepackage[utf8]{inputenc}
\usepackage[T1]{fontenc}
\usepackage{amsmath, amsfonts, amssymb, amsthm}
\usepackage{hyperref}
\usepackage{orcidlink}
\newtheorem{thm}{Theorem}
\newtheorem{lem}{Lemma}
\newtheorem{cor}{Corollary}
\newtheorem{defn}{Definition}

\newtheorem{propn}{Proposition}
\newtheorem{rem}{Remark}

\date{}
\DeclareMathOperator{\tr}{tr}
\begin{document}
		\author
		{Tanay Kochrekar \thanks{tanaykochrekar@gmail.com}
		\quad \quad 
		A. S. Garge \thanks{anuradha.garge@gmail.com} \\
		\textit{\scriptsize{\parbox{\textwidth}{\centering Department of Mathematics, University of Mumbai, Vidyanagari, Santacruz East, Mumbai, Maharashtra, 400098, India.}}}}
	\title{MATRICES OVER NON-COMMUTATIVE RINGS AS SUMS OF FIFTH AND SEVENTH POWERS}
	\date{}
	
	\maketitle	
	\begin{abstract}
		Let $R$ be non-commutative ring with unity. In this paper, we prove that a given matrix in $M_n(R)$ is a sum of $p$-th powers in $M_n(R)$, where $n < p$ if and only if its trace can be written as a sum of $p$-th powers in $M_n(R)$, for $n = 2,3,4$ and $p = 5$ and $n = 4,5,6$ for $p = 7$. This paper extends the results of S. A. Katre and Kshipra Wadikar, $(p=3)$ and Garge $(n \geq p), (p = 5, 7)$, and S. A. Katre and Deepa Krishnamurthi $(n \geq p)$.
	\end{abstract}
	\textbf{Mathematics Subject Classification (2020)}: Primary 11P05; 15A30;
		Secondary 16S50
		
	\textbf{Key words and phrases:} Matrices; Non-commutative rings; trace; Sums of powers;	Waring’s problem
	
	\section{Introduction}
	Expressing a matrix as sum of powers of matrices has always been an active area of research, being led by Carlitz, Newman, Vaserstein, Wadikar, Katre and Richman. Katre and Krishnamurthi\cite{KK} have obtained a result for a matrix over a non-commutative ring to be sum of $p$-th powers when $n \geq p \geq 2$ where $n$ is order of the matrix and $p$ is a prime. Katre-Wadikar \cite{KW} have obtained conditions for the case where $n = 2$, $p = 3,4$. Conditions for matrices over commutative rings to be sums of fifth and seventh powers are obtained by Garge \cite{G}, Garge and Barai \cite{BG} for sixth and eighth powers, Muangma and Rodtes \cite{MR} from ninth to sixteenth power. The problem remains open for the order of matrix over non-commutative rings, $n < p$ and for larger values of the prime power $p$. In this paper, we try to obtain conditions for the case where $n < p$ and $p = 5$ and $p=7$. We will try to write $p$-th powers of matrices having order less than prime $p$ as sums of $p$-th powers of matrices. 
	
	Throughout this paper, $R$ is a non-commutative ring with unity.
	We begin by quoting some results by Katre and Krishnamurthi \cite{KK}.

	\begin{defn}
		Let $\sigma$ be the permutation $(12345)$ in cyclic notation on the symbols $1,2,3,4,5$. Clearly, $(12345)^5 = (1)$.
		
		Define:
		$C: R^5 \rightarrow R$ by
		$C(a_1,a_2,a_3,a_4,a_5) = \sum_{k=1}^{5} \sigma^{k-1}(a_1a_2a_3a_4a_5) = a_1a_2a_3a_4a_5 + a_5a_1a_2a_3a_4 + a_4a_5a_1a_2a_3 + a_3a_4a_5a_1a_2 + a_2a_3a_4a_5a_1.$
		This will be called as a `cyclic sum' over $5$ elements.\end{defn}

	The next proposition which we recall, comments on a relation between cyclic sums and commutators. A commutator is defined for $a,b \in R$ by $[a,b] = ab - ba.$
	
	\begin{propn}{\rm [\cite{KK}, Proposition 2.1]} \label{p1}
		For $a_{1}, a_{2}, \ldots, a_{k} \in R$, the cyclic sum $C\left(a_{1}, a_{2}, \ldots, a_{k}\right)$ is a sum of commutators modulo $k R$.
	\end{propn}
	\begin{proof} Observe that $a_{1} a_{2} \cdots a_{k}+a_{2} a_{3} \cdots a_{k} a_{1}+\cdots+a_{k} a_{1} a_{2} a_{3} \cdots a_{k-1}=\left(a_{2} a_{3} \cdots a_{k}\right)$ $a_{1}-a_{1}\left(a_{2} a_{3} \cdots a_{k}\right)+\left(a_{3} a_{4} \cdots a_{k}\right)\left(a_{1} a_{2}\right)-\left(a_{1} a_{2}\right)\left(a_{3} a_{4} \cdots a_{k}\right)+\cdots+a_{k}\left(a_{1} a_{2} \cdots\right.$
		$\left.a_{k-1}\right)-\left(a_{1} a_{2} \cdots a_{k-1}\right) a_{k}+k a_{1} a_{2} \cdots a_{k}=\left[a_{2} a_{3} \cdots a_{k}, a_{1}\right]+\left[a_{3} a_{4} \cdots a_{k}, a_{1} a_{2}\right]+\cdots+$ $\left[a_{k}, a_{1} a_{2} \cdots a_{k-1}\right]+k a_{1} a_{2} \cdots a_{k}$.
		
		This proves that $C\left(a_{1}, a_{2}, \ldots, a_{k}\right)$ can be expressed as sum of commutators, modulo $kR$.
	\end{proof} 
	We also consider the action of the cyclic group generated by the $k$-cycle
	$\sigma=(1,2, \ldots, k)$, on the set of $k$-tuples of elements of a set by $\sigma\left(a_{1}, a_{2}, \ldots, a_{k}\right) = \left(a_{2}, a_{3}, \ldots, a_{k}, a_{1}\right)$. If $k = p$ is a prime, by orbit-stabilizer theorem, the number of elements in the orbit of any $p$-tuple divides $p,$ which is the order of the group $\langle\sigma\rangle$. Here, the orbit has $1$ or $p$ elements. Therefore, if at least two of $a_{1}, a_{2}, \ldots, a_{p}$ are unequal, the orbit has exactly $p$ elements.
	
	\begin{propn}{\rm [\cite{KK}, Proposition 2.2]} \label{p3} If $R$ is a non-commutative ring and $n \geq p \geq 2,$ $p$ prime, then for $A \in$ $M_{n}(R)$, trace $\left(A^{p}\right)$ is the sum of pth powers of diagonal elements of $A$ and cyclic sums $C\left(a_{1}, a_{2}, \ldots, a_{p}\right)$ with $a_{1}, a_{2}, \ldots, a_{p} \in R$.
	\end{propn}
	
	Now the main theorem by Katre-Krishnamurthi \cite{KK}, which states some conditions for an $n \times n$ matrix to be a sum of $p$-th powers of matrices.
	
	\begin{thm}{\rm [\cite{KK}, Theorem 1]}{\label{t1}}
		Consider positive integers $n$ and $p$, such that $n \geq p \geq 2$ and $p$ is a prime. Let $T_p = T_{p,n}$ be the set of those elements of R that can be expressed as sums of traces of $p$-th powers over $M_n(R)$. Let $S$ be a set of representatives of orbits of $p$-tuples of elements of $\{c_1, c_2, \ldots , c_l\} \subseteq R, l \geq p$ and let $S^\prime$ be the set of such representatives in which we have at least two unequal entries. Then, 
		\begin{enumerate}
			\item[(i)] For $a, a_1, a_2, \cdots a_p$, the cyclic sum $C(a_1, a_2 \cdots a_p) \in T_p$. Also $pa \in T_p$, $a^p \in T_p$.
			\item[(ii)] $T_{p}=\biggl\{\sum_{(a_{1}, a_{2}, \ldots, a_{p}) \in S} C \left(a_{1}, a_{2}, \ldots, a_{p}\right) + \sum_{j=1}^{l} c_{j}^{p} \mid l \geq 1, a_{i}, c_{j} \in R, 1 \leq i \leq p, 1 \leq j \leq l\biggr\}.$ 
			\item[(iii)] $T_{p}=\biggl\{\displaystyle \sum_{\left(a_{1}, a_{2}, \ldots, a_{p}\right) \in S} C\left(a_{1}, a_{2}, \ldots, a_{p}\right)+c^{p} \mid a_{i}, c \in R, 1 \leq i \leq p\biggr\}$ where $S$ is a finite subset of $R^p$.
			\item[(iv)] $T_{p}=\biggl\{\displaystyle \sum_{j=1}^{q}\left(a_{j} b_{j}-b_{j} a_{j}\right)+\sum_{j=1}^{l} c_{j}^{p}+p r \mid a_{j}, b_{j}, c_{j}, r \in R, q \geq 1, l \geq 1\biggr\}$.
			\item[(v)] $T_{p}=\biggl\{\displaystyle \sum_{j=1}^{q}\left(a_{j} b_{j}-b_{j} a_{j}\right)+c^{p}+p r \mid a_{j}, b_{j}, c, r \in R, q \geq 1, l \geq 1\biggr\}$.
			\item[(vi)] A matrix $A \in M_{n}(R)$ is a sum of $p$-th powers in $M_n(R)$ if and only if trace(A) is a sum of $p$-th powers and commutators modulo $p R$ if and only if trace $(A)$ is a sum of a $p$-th power and commutators modulo $p R$.
			\item[(vii)] Vaserstein (\cite{V}, Theorem 1): A matrix $A \in M_{n}(R)$ is a sum of squares if and only if trace $(A)$ is a sum of squares modulo $2 R$.
		\end{enumerate}
	\end{thm}
	The following two theorems act as a base for all of the following investigation. We need to simplify the expression for sums of traces of $k$-th powers of matrices. In general, this is difficult and has been done for $n = 2$ and $k = 3,4.$ by Katre and Wadikar,\cite{KW}. We try a similar simplification for $n = 2,3,4,5$ when $k = 5$ and $n = 2,3,4,5,6$ when $k = 7.$ Given below is theorem \ref{strthm} that establishes the standard form of an element in the power of a matrix.   
	\begin{thm}{\rm [\cite{KW}, Theorem 3.2]}\label{t2}
		Let $n, k \geq 2$ be integers and $A \in M_n(R)$. Then $A$ is a sum of
		$k$-th powers of matrices in $M_n(R)$ if and only if $\tr(A)$ is a sum of traces of $k$-th powers of matrices in $M_n(R)$.
	\end{thm}
	
	The one-way implication is obvious as trace is additive. 
	For the converse, suppose $\tr(A)$ is sum of traces of $k$-th powers of matrices in $M_n(R)$. We have to study the structure of an arbitrary trace of the $k$-th power of matrix. For the cases where $n>p$ where $p$ is a prime, theorem $\ref{t1}$ holds, and gives us an idea of such traces. We explore the cases where $n < p$, and try to derive similar structure for such traces and some specific values of $n$ and $p$.

	\begin{thm}\label{strthm}
		Let $A \in M_n(R)$. Let $A = [a_{ij}], 1 \leq i, j \leq n$. Consider $A^p = [a_{ij}^{(p)}]$, for a prime number $p$. Then
		$$a_{ij}^{(p)} = \sum_{k_{1}=1}^n \sum_{k_{2}=1}^n\ldots \sum_{k_{p-1}=1}^n a_{i{k_1}}a_{{k_1}{k_2}}\ldots a_{k_{p-1}j} =  \sum_{1 \leq k_1,k_2,\ldots k_p \leq n} a_{i{k_1}}a_{{k_1}{k_2}}\ldots a_{k_{p-1}j}$$
		Summarizing, for any $n, p \in \mathbb{N}$, $p$ prime, the elements in the trace of $A^p$ are sums of $p$-cycles and $p$-th powers of elements.
		Therefore, if possible, it suffices to prove the `most arbitrary' element in the trace of the $A^p$ is contained in $T_{p,n}$.
	\end{thm}
	\begin{proof} We use induction on the power, $m$ for a fixed order $n$. For $m=1$, the claim is clear. 
		
		For $n = m = 2$, 
		$A^2 = \begin{bmatrix}
			a_{11}^2 + a_{12}a_{21} & a_{11}a_{12}+a_{12}a_{22}\\
			a_{21}a_{11}+a_{22}a_{21} & a_{21}a_{22}+a_{22}^2
		\end{bmatrix}
		$ it can be seen that every entry consists of two terms, where the suffixes follow the claimed pattern.
		For $n=2$ and $m = 3$,
		
		$A^3 = [a_{ij}^{(3)}]$ where
		$a_{11}^{(3)} = a_{11}^3+a_{11}a_{12}a_{21}+a_{12}a_{21}a_{11}+a_{12}a_{22}a_{21}$,
		
		$a_{12}^{(3)} = a_{11}^2a_{12}+a_{12}a_{22}^2+a_{11}a_{12}a_{22}+a_{12}a_{21}a_{12},$
		
		$a_{21}^{(3)} = a_{21}a_{11}^2+a_{22}^2a_{21}+a_{21}a_{12}a_{21}+a_{22}a_{21}a_{11},$ 
		
		$a_{22}^{(3)}= a_{22}^3+a_{21}a_{11}a_{12}+a_{21}a_{12}a_{22}+a_{22}a_{21}a_{12}.$
		
		Let the statement be true for $r$. Then $s$-th row will have the elements $a_{s1}^{(r)}, a_{s2}^{(r)}, a_{s3}^{(r)}, \ldots , a_{sn}^{(r)}$ and $t$-th column will have the elements $a_{1t}^{(r)}, a_{2t}^{(r)}, \ldots, a_{nt}^{(r)}$
		Here,
		$a_{s1}^{(r)} = \displaystyle \sum_{k_1, k_2,\ldots, k_{m-1
		}} a_{s{k_1}}a_{{k_1}{k_2}}\ldots a_{k_{r-2}k_{r-1}}a_{k_{r-1}1}$
		and
		
		$a_{3t}^{(r)} = \displaystyle \sum_{k_1, k_2,\ldots, k_{r-1
		}} a_{3{k_1}}a_{{k_1}{k_2}}\ldots a_{k_{r-2}k_{r-1}} a_{k_{r-1}t}$ etc.
		
		Consider the element in $s$-th row and $t$-th column in $A^{r+1}$. We have,
		
		\begin{flushleft}
			$a_{st}^{(r+1)} = a_{s1}^{(r)}.a_{1t} + a_{s2}^{(r)}a_{2t} + \cdots + a_{sn}^{(r)}a_{nt}. $
			
			$= \left(\sum_{k_1, k_2,\ldots, k_{r-1
			}} a_{s{k_1}}a_{{k_1}{k_2}}\ldots a_{k_{r-2}k_{r-1}} a_{k_{r-1}1}\right)a_{1t}
			+ \left(\sum_{k_1, k_2,\ldots, k_{r-1}} a_{s{k_1}}a_{{k_1}{k_2}}\ldots a_{k_{r-2}k_{r-1}} a_{k_{r-1}2}\right)a_{2t}
			+ \cdots + \left(\sum_{k_1, k_2,\ldots, k_{r-1}} a_{s{k_1}}a_{{k_1}{k_2}}\ldots a_{k_{r-2}k_{r-1}} a_{k_{r-1}n}\right)a_{nt}$ \\
			
			$=\left(\sum_{k_1, k_2,\ldots, k_{r-1
			}} a_{s{k_1}}a_{{k_1}{k_2}}\ldots a_{k_{r-1}1}a_{1t}\right)	+ $\\
			
			$\left(\sum_{k_1, k_2,\ldots, k_{r-1
			}} a_{s{k_1}}a_{{k_1}{k_2}}\ldots a_{k_{r-1}2}a_{2t}\right) + \cdots + \left(\sum_{k_1, k_2,\ldots, k_{r-1
			}} a_{s{k_1}}a_{{k_1}{k_2}}\ldots a_{k_{r-1}n}a_{nt}\right)$\\
			
		\end{flushleft}
		There are $r+1$ terms in each sum, and therefore the result can be summarized as
		$$
		a_{st}^{(r+1)} = \sum_{k_1, k_2,\ldots, k_{r
		}} a_{s{k_1}}a_{{k_1}{k_2}}\ldots a_{{k_{r-1}}{k_r}} a_{k_{r}t}, 1 \leq k_i \leq n.	
		$$
		By induction, the result holds for all $m \in \mathbb{N}$.
	\end{proof}
	
		By `most arbitrary' choice of cycle, we mean the cycles that have the least number of entries repeated, or alternatively, a $p$-cycle such that any other $p$-cycle can be obtained by substituting appropriate variables in these cycles. For example, for $n = 2$ and $p = 5$, the `most arbitrary' cycles are $C(a,a,b,c,d), C(a,b,a,c,d)$, since $n^2 < p$. If $n^2 > p$ then the most arbitrary cycle would be the one consisting of all different variables, $C(a_1, a_2, \ldots a_p)$. If one proves that the most arbitrary cycles are contained in the $T_{n,p}$ under discussion, then $p$-th power of any matrix can be expressed as sums of $p$-th powers of matrices, by Theorem \ref{t2}.

	\begin{cor}\label{orbcor}
		Consider the group action of the cyclic group generated by cycle of length $p$, where $p$ is the power of the matrix, on any term appearing in the trace of the $p$-th power of a matrix. Then the trace will contain the entire `orbit' of this element under the said group action. 
	\end{cor}
	\begin{proof}
		From Theorem \ref{strthm}, every summand in the trace of the $p$-th power of the matrix will have the form $a_{ii}^{(p)} = \sum_{k_1, k_2,\ldots, k_{p-1
		}} a_{i{k_1}}a_{{k_1}{k_2}}\ldots a_{k_{p-2}k_{p-1}}a_{k_{p-1}i}$. Note that the first suffix of the first element and the last suffix of the last element in each term will be the same. It is easy to see that when the cyclic permutation $(1,2,3,\ldots, p)$ acts on a term from the trace, the resulting term is also present in the trace. 
		Suppose $	a_{jj}^{(r+1)} = \sum_{k_1, k_2,\ldots, k_{p
		}} a_{j{k_1}}a_{{k_1}{k_2}}\ldots a_{{k_{p-1}}j}$ is such a term. When the permutation $(1,2,3,\ldots, p)$ acts on it, we get $a_{{k_1}{k_2}}\ldots a_{{k_{p-1}}j}a_{j{k_1}}$ which is an element from $k_1$th row and $k_1$th column, hence present in the trace. Similarly, if we let the cyclic permutation act on the element $a_{{k_1}{k_2}}\ldots a_{{k_{p-1}}j}a_{j{k_1}}$, the image element $a_{{k_2}{k_3}}a_{{k_3}{k_4}}\ldots a_{{k_{p-1}}j}a_{j{k_1}}a_{{k_1}{k_2}}$ will appear in the trace. Following in this fashion, the $p$-th such element will be the original element and thus the entire orbit will be present in the trace.
	\end{proof}
	
	\begin{cor}
		If a matrix of order $n$ is raised to the power $p$ where $p$ is a prime, then the `cyclic sum' corresponding to an element in the trace will be present in the trace, and there will be either $p$ or $1$ terms in the cyclic sum. If the power is a composite number then the orbit of an element in the trace will be divisor of $r$.
	\end{cor}
	\begin{proof}
		From the Corollary \ref{orbcor} it is established that the cyclic group of order $m$ generated by $(1,2,3,\ldots, m)$ will act on a term in the trace and the corresponding orbit will lie in the trace. 
		
		By orbit-stabilizer theorem, the size of the orbit will divide the order of group, and the corollary follows. For $m = p$, where $p$ is prime, the elements with orbit size $1$ must be the elements of the type $a_{jj}^p$, for all $j$, $1 \leq j \leq n$ which are $p$-th powers of the diagonal entries. The elements with orbit size $p$ must therefore produce all distinct elements. For $m = r$ where $r$ is composite, the orbit of size $d$ will be stabilized by a subgroup of the order $\frac{r}{d}$.
	\end{proof}
	
	We now use the method by Katre, [\cite{KG}, Theorem 3.5] to show that the matrices that are sums of $p$-th powers over $M_n(R)$ can be expressed as sums of $p$-th powers over $M_{n+1}(R)$ as well.
	\begin{lem}\label{subgrouplem}
		Let $p$ be a prime. The set of elements of $R$ that can be expressed as sums of traces of $p$-th powers over $M_n(R)$, $T_{p,n}$ is an additive subgroup $H \subseteq R$. Also, $T_{p,n} \subset T_{p,n+1}$ for all $n \in \mathbb{N}$.
	\end{lem}
	\begin{proof}
		As $\tr O^p  = 0$, clearly $0 \in T_{p,n}$. Let $A_i, B_j \in M_n(R)$. If $a = \displaystyle \sum_{i=1}^{s} \tr (A_i^p)$ and $b = \displaystyle \sum_{i=1}^{t} \tr (B_i^m)$ then $a - b = \displaystyle \sum_{i=1}^{s} A_i^p + \displaystyle \sum_{i=s+1}^{s+t} (-B_{i})^{m}.$ 
		
		Note that $a = \displaystyle \sum_{i=k}^{s} \tr (A_k^p)$. Let $A_{k} = [a_{kij}]$. 
		Consider 
		${A_i^\prime} = [a^\prime_{kij}], 1\leq i, j \leq n+1 $, where $a^\prime_{kij} = a_{kij}$ for $1\leq i, j \leq n$ and $0$ otherwise. We use the same technique as \cite{KK}; to augment $A_k$ by a zero row and zero column to yield	$a = \displaystyle \sum_{k=1}^{s} \tr ({A_k^\prime}^p)$. Thus $T_{p,n} \subset T_{p,n+1}$. 
	\end{proof}
	\begin{cor}\label{coprimecor} Let $R$ be a non-commutative ring with unity and $a \in R$. If $ua \in T_{p,n}$ and $va \in T_{p,n}$ where $u, v$ are coprime integers, then $a \in T_{p,n}$.
	\end{cor}
	\begin{proof}
		If $u, v$ are coprime then there are $x, y \in \mathbb{Z}$ such that $ux + vy = 1.$ Thus $uxa + vya = a$. Since $ua \in T_{p,n}$ it is clear that $xua \in T_{p,n}$ by repetitive addition and subtraction for all $x \in \mathbb{Z}$. Similarly, $vy \in T_{p,n}$, hence we get $a \in T_{p,n}$.\end{proof}
	
	\section{The prime $p = 5$.}
	In this section, we try to find conditions for matrices over non-commutative rings to be sums of fifth powers of matrices. For the rings $M_3(R)$ and $M_4(R)$, we use the fact that the most arbitrary five-cycle can be obtained as sums of fifth powers. For the case of $M_2(R)$, results for matrices over commutative rings obtained by Garge \cite{G} are used to obtain the results, noting that the elements used in the proof commute even in any non-commutative ring with unity.
	
	\subsection{Order of matrices, $n=2$ and prime power, $p = 5$.}
	First we classify the most general five-cycles which can be made using $4$ elements, upto rearrangement of choices of the four elements coming from a  $2 \times 2$ matrix.
	
	\begin{rem}\label{r52}
		Observing trace of a fifth power of a matrix over $M_2(R)$ helps us to note some observations. 
		
		If $A = \begin{bmatrix}
			x & y\\
			z & w\\
		\end{bmatrix}$, then $\tr (A^5) = x  y z y z + y  w z y z+ x^3 y z + y  z  x y z  + x  y  w^2 z + y  w^3 z + x^2 y w z + y  z  y w z  
		+  x  y  z x^2 + y  w  z x^2 + x^5 + y  z  x^3   
		+ x  y  w z x + y  w^2 z x + x^2  y z x + y  z  y z x + w  z  y  z  y + z  x  y  z  y + w^3 z y + z  y  w  z  y  
		+ w  z x^2 y + z  x^3 y + w^2 z x y + z  y  z x y + w  z  y  w^2  + z  x  y  w^2 + w^5 + z  y  w^3  
		+ w  z  x  y  w + z  x^2  y  w + w^2 z y w + z  y  z y w =  x^5 + w^5 + 
		C(x,y,z,y,z) + 
		C(w,z,y,z,y)+
		C(x,x,x,z,y) +  
		C(x,x,y,w,z) +
		C(x,y,w,w,z) + 
		C(y,w,w,w,z)$.
		
		It is clear by the pigeonhole principle that, only two main types of cyclic sums over five symbols exist:
		\begin{enumerate}
			\item[(A)] $C(a,b,c,d,a) = abcda + bcda^2 + cda^2b + da^2bc + a^2bcd.$
			\item[(B)] $C(a,b,a,c,d) = abacd + bacda + acdab + cdaba + dabac.$
		\end{enumerate}
		Since at a later point, we require entries in the form of $a_{ij}$ where $i, j \in \{1, 2\}$, and $i \neq j$, we can rewrite:
		\begin{enumerate}
			\item 	$C(a_{12},a_{21},a_{12},a_{22},a_{21})$ and $C(a_{21},a_{12},a_{21},a_{11},a_{12})$
			\item   $C(a_{12},a_{21},a_{12},a_{21},a_{11})$ and $C(a_{21},a_{12},a_{21},a_{12},a_{22})$
		\end{enumerate}
		Since five-cycles with at least one repeated element were observed in the calculations, we define a set $S_{5,2} \subset R^5$ as the set of all the five-cycles of the types described here. We note that any other five-cycle with four symbols will definitely belong to one of the types listed here.
	\end{rem}
		
	Now we state a lemma that is useful to combine various fifth powers into a fifth power and sums of few cyclic sums.
	\begin{lem}\label{lbi5}
		Let $R$ ring be a non-commutative ring with unity. Sums of fifth powers of elements can be represented as sums of certain cyclic sums and a fifth power.
	\end{lem}
	\begin{proof}
		We first prove the claim for the fifth power of sum of two elements. Let $a,b \in R$.
		
		Then $(a+b)^5 = a^5 + a^4b + a^3ba + a^3b^2 + a^2ba^2 + a^2bab + a^2b^2a + a^2b^3 + aba^3 + aba^2b + ababa + abab^2 + ab^2a^2 + ab^2ab + ab^3a + ab^4 + ba^4 + ba^3b + ba^2ba + ba^2b^2 + baba^2 + babab + bab^2a + bab^3 + b^2a^3 + b^2a^2b + b^2aba + b^2ab^2 + b^3 a^2 + b^3ab + b^4a + b^5.$
		
		Rearranging so that the cyclic sums are together:
		We have, $(a+b)^5  = a^5 + b^5 + (a^4b + a^3ba + a^2ba^2 + aba^3 + ba^4)
		+ (a^3b^2 + a^2b^2a + ab^2a^2 + b^2a^3 + ba^3b)
		+ (a^2bab + ababa + baba^2 + aba^2b + ba^2ba)
		+ (a^2b^3 + ab^3a + b^3a^2 + b^2a^2b + ba^2b^2)
		+ (abab^2 + bab^2a + ab^2ab + b^2aba +  babab)
		+ (ab^4 +  b^3ab + + b^2ab^2 + bab^3 + b^4a) $
		
		$= a^5 + C(a,a,a,a,b) + C(a,a,a,b,b) + C(a,a,b,a,b)$ $+ C(a,a,b,b,b) + C(a,b,a,b,b) + C(a,b,b,b,b) + b^5$.	
				
		Thus $a^5 + b^5 = (a+b)^5 + C$ where $C = C(a,a,a,a,-b) + C(a,a,a,b,-b) + C(a,a,b,a,-b)$ $+ C(a,a,b,b,-b) + C(a,b,a,b,-b) + C(a,b,b,b,-b)$. We proceed by induction on $n$. Consider $\displaystyle \sum_{i=1}^{k}a_i^5$ where $a_i \in R,$ for all $1 \leq i \leq n$.  Let the proposition hold for $n = k$, i.e. $\displaystyle \sum_{i=1}^{k}a_i^5 = C + b^5$ where $C$ is sum of some cyclic sums and $b = \displaystyle \sum_{i=1}^5 a_i$.
		
		Consider $\displaystyle \sum_{i=1}^{k+1} a_i^5 = \sum_{i=1}^{k}a_i^5 + a_{k+1}^5 = C + b^5 + a_{k+1}^5 = C + (b+a_{k+1})^5 + C^\prime $
		
		where $C^\prime = C(a_{k+1},a_{k+1},a_{k+1},a_{k+1},-b) + C(a_{k+1},a_{k+1},a_{k+1},b,-b) + C(a_{k+1},a_{k+1},b,a_{k+1},-b) + C(a_{k+1},a_{k+1},b,b,-b) + C(a_{k+1},b,a_{k+1},b,-b) + C(a_{k+1},b,b,b,-b)$
		where $C$ and $C^\prime$ are some cyclic sums, and result follows by induction. It should be noted that all the $5$-cycles here are special cases of the types (A) or (B) described above.
	\end{proof}
	
	Now we state an important result by Garge \cite{G} which uses only the matrices with entries $\pm 1$, $r + 1$, $r$, $r^k$, for some $k \in \mathbb{N}$,  which commute with each other for any element $r \in R$. Although they have discovered this for commutative ring, we can conclude due to the nature of entries which are polynomials in the ring $\mathbb{Z}[r]$, the same holds over non-commutative rings as well.
	
	\begin{lem}[\cite{G}, Lemma 2.5]\label{pg5r}
		Let $R$ be a ring, not necessarily commutative; with units, and $r \in R$. Any element of the form $5r$ can be written as a sum of traces of fifth powers over $M_2(R)$.
	\end{lem}

	Thus, if a matrix in $M_2(R)$ is a sum of fifth powers of matrices over $M_2(R)$, the trace can be written sum of a fifth powers of diagonal elements and specific cyclic sums. By Lemma \ref{lbi5}, the fifth powers can be combined into a fifth power and a sum of the rest of the terms are cyclic sums.
	
	\begin{lem}\label{lcalc52}
		The main two types of five-cycles as given in the remark \ref{r52} $C(x,y,z,y,z), \ C(w,w,z,x,y) \in T_{5,2}.$
	\end{lem}
	\begin{proof}
	Let $A = \begin{bmatrix}
			x & y\\
			z & w\\
		\end{bmatrix}$ 	and $C = \begin{bmatrix}
			x & y\\
			z & 0 \\	
		\end{bmatrix}$. 
		Then we have that $\tr(A^5 + (-C^5)) = w^5+
		C(w,w,w,z,y)+ 
		C(w,z,x,x,y)+ 
		C(w,w,z,x,y)+ 
		C(w,z,y,z,y)$. Finally, if $D = \begin{bmatrix} -x & y\\z & w	\end{bmatrix}$, then
		$\tr(A^5 + D^5) = 2(w^5 +
		C(w,w,w,z,y)+
		C(w,z,x,x,y)+
		C(w,z,y,z,y)).$ Thus, $\tr (A^5) - \tr(D^5) + \tr(-C)^5 + \tr(-C)^5 = \tr (A^5 + (-D^5) + (-C^5) + (-C^5)) = 2C(w,w,z,x,y)$ By Proposition \ref{pg5r}, there exists $M \in M_2(R)$ such that $\tr(M^5) = 5C(w,w,z,x,y)$, so
		$\tr (N^5) - 2(\tr (A^5) + \tr(-D^5) + \tr(-C^5) + \tr(-C^5)) = C(w,w,z,x,y)$.
		
		Consider $P = \begin{bmatrix}w & z\\y & 0	\end{bmatrix}$ and $T = \begin{bmatrix}w & -z\\y & 0\end{bmatrix}$, then $\tr(P^5) + \tr(T^5) = 2 w^5 + 2C(w, z,  y,  z,  y).$ Since $w^5 \in T_{5,2}$ , we yield $2C(w,z,y,z,y) \in T_{5,2}$, by Lemma \ref{subgrouplem}. Thus $C(w,z,y,z,y) \in T_{5,2}$ using Corollary \ref{coprimecor}.
	\end{proof}
	
	\begin{thm}\label{t52}
		Let $R$ be a non-commutative ring with unity.
		Let $T_{5,2}$ be the set of those elements of $R$ that can be expressed as sums of traces of fifth powers over $M_2(R)$. The set $S_{5,2}$ is defined in the remark \ref{r52}.
		\begin{enumerate}
			\item[(i)] $C(w,w,w,x,y), C(x,y,z,y,z), C(x,x,y,w,z)  \in T_{5,2}$. Also $5a \in T_{5,2}$, $a^5 \in T_{5,2}$.
			\item[(ii)]$ T_{5,2}=\biggl\{\displaystyle \sum_{\left(a_{1}, a_{2}, \ldots, a_{5}\right) \in S_{5,2}} C\left(a_{1}, a_{2}, \ldots, a_{5}\right)+\sum_{j=1}^{l} c_{j}^{5} \mid a_i, c_j \in R, 1 \leq i \leq p, 1 \leq j \leq l\biggr\}.$
			\item[(iii)] $T_{5,2}=\biggl\{\displaystyle \sum_{\left(a_{1}, a_{2}, \ldots, a_{5}\right) \in S_{5,2}} C\left(a_{1}, a_{2}, \ldots, a_{5}\right)+c^{5} \mid 1 \leq i \leq p, 1 \leq j \leq l \biggr\}.$	\end{enumerate}
	\end{thm} 
	\begin{proof}
		\begin{enumerate}
			\item[(i)] Follows from Lemma \ref{lcalc52}. Also, $\tr \begin{bmatrix}
				a & 0 \\ 0 & 0
			\end{bmatrix}^5 = a^5$.  $5a \in T_{5,2}$ by Proposition \ref{pg5r}	.		
			\item[(ii)] Follows from Theorem \ref{strthm}  and (i).
			\item[(iii)] Follows from Lemma \ref{lbi5} and (ii).
		\end{enumerate}
	\end{proof}
	
	To conclude, we state:
	\begin{thm}\label{t52final} Let $R$ be a non-commutative ring with unity. Then, $A \in M_2(R)$ is a sum of fifth powers over $M_2(R)$ if and only if $\tr(A)$ is a sum of fifth powers and some specific five-cycles from $S_{5,2}$.
	\end{thm}
	\begin{proof}
		Suppose $A = \displaystyle \sum_{i=1}^{m} M_i^5, M_i \in M_2(R)$ for all $1\leq i \leq m.$ By Theorem \ref{strthm}, trace of $A$ will be sums of $5$-cycles. We observe that $\tr(A) = \displaystyle \sum_{i=1}^{m} \tr(M_i^5) = \displaystyle \sum_i (a_i^5 + b_i^5 + C_i)$ where $a_i, b_i$ are the diagonal elements of $M_i$ and $C_i$ are the  five-cycles described as type (A) and (B) from Remark \ref{r52}. By Lemma \ref{lbi5}, $\displaystyle \sum a_i^5 + b_i^5$ can be combined into a single fifth power and more cyclic sums, say $C^\prime.$
				
		Hence $\tr(A) = \displaystyle \sum_{i=1}^{m} M_i^5 = \alpha^5 + \tilde{C}$ for some $\alpha \in R$, and $\tilde{C} = \sum C_i + C^\prime$, essentially sums of five-cycles from $S_{5,2}$.
		
		For the converse, suppose $\tr(A) = \alpha^5 + \tilde{C}$, where $\tilde{C}$ is the sums of specific five-cyclic sums strictly from $S_{5,2}$.
		
		By Theorem \ref{t52}, there exist $N_i \in M_2(R)$ such that $\displaystyle \sum_{i=1}^t \tr(N_i^5) = \tilde{C}$. Clearly $N = \begin{bmatrix}
			a & 0 \\ 0 & 0
		\end{bmatrix}$ satisfies $\tr(N^5) = \alpha^5$.
		So, $\tr(A) = \tr (\displaystyle \sum_{i=1}^{s+1} N_i^5)$, where $N_{s+1} = N$.
		Finally, we appeal to Theorem \ref{t2} to conclude that $A = \displaystyle \sum_{i=1}^{r} N_i^{\prime \ 5}$ for some $N_i^{\prime \ 5} \in M_2(R)$. 
	\end{proof}
		
	\subsection{Order of matrices, $n=3, 4$ and prime power, $p = 5$.}
	In this section, we begin by obtaining the most general five-cycle with as sums of traces of fifth powers of matrices over $M_3(R).$ Doing this enables us to immediately establish a version of Theorem \ref{t52} for $n =3$ as well as $n = 4$.
	\begin{defn}\label{S53def}
	The set $S_{5,3}$ is representatives of orbits of $5$-tuples of elements of ${r, s, t, u, v, w, x, y, z} \subseteq R$, and $S^\prime_{5,3}$ is set of such representatives in which we have at least two unequal entries.
	\end{defn}
	
	\begin{lem}\label{lcalc53}
		Any five-cyclic sum can be written as sums of trace of a fifth powers of matrices in $M_3(R)$. i.e. similar to above notation, $C(a,b,c,d,e) \in T_{5,3}$, where $a, b, c, d, e \in R$.
	\end{lem}
	\begin{proof}
		Consider the matrices:
		$D = \begin{bmatrix}
			0 & a& 0\\ e & 0 & b\\ 0 & d & c\\
		\end{bmatrix},
		F = \begin{bmatrix}
			-c & d & 0\\ b & 0 & 0\\ 0 & 0 & 0\\
		\end{bmatrix} $
		
		Then $\tr(D^5) = abcde + bc^3d + bcdbd + bcdea + bdbcd + c^5 + 
		c^3db + c^2dbc + cdbc^2 + cdbdb + cdeab + dbc^3 + dbcdb + dbdbc + 
		deabc + eabcd = c^5 + C(a,b,c,d,e) +$\\
		$C(c,d,b,d,b) + C(c,c,c,d,b)$ and $\tr(F^5) = -bc^3d - bcdbd - bdbcd - c^5 - c^3db - c^2dbc - cdbc^2 - cdbdb - dbc^3 - dbcdb - dbdbc = - [c^5 + C(c,c,c,d,b) + C(c,d,b,d,b)]$.
		
		It is now clear that $C(a,b,c,d,e) = abcde + bcdea + cdeab + deabc + eabcd = \tr(D^5) + \tr(F^5).$ By Lemma \ref{subgrouplem}.
	\end{proof}
	
	\begin{cor}\label{c5r53}
		Any element of the form $5r$ for any $r \in R$ can be expressed as sums of traces of fifth powers of matrices in $M_3(R)$.
	\end{cor}
	\begin{proof}
		Observe that $5r = C(r,1,1,1,1)$. Now Lemma \ref{lcalc53} applies.
	\end{proof}
	\begin{thm}\label{t53}
		Let $R$ be a non-commutative ring with unity.
		Let $T_{5,3}$ be the set of those elements of $R$ that can be expressed as sums of traces of fifth powers over $M_n(R)$, where $n = 3$ or $n = 4$.
		
		\begin{enumerate}
			\item[(i)] Any arbitrary $5$-cycle $C(a,b,c,d,e)\in T_{5,n}$. Also $5r, a^5 \in T_{5,n}$.
			\item[(ii)] $T_{5,n}=\biggl\{\displaystyle \sum_{\left(a_{1}, a_{2}, \ldots, a_{5}\right) \in S_{5,3}} C\left(a_{1}, a_{2}, \ldots, a_{5}\right)+\sum_{j=1}^{l} c_{j}^{5} \mid l \geq 1, a_{i}, c_{j} \in R, 1 \leq i \leq 5, 1 \leq j \leq l\biggr\} .$
			\item[(iii)] $T_{5,n}=\biggl\{\displaystyle \sum_{\left(a_{1}, a_{2}, \ldots, a_{5}\right) \in S_{5,3}} C\left(a_{1}, a_{2}, \ldots, a_{p}\right)+c^{5} \mid a_i \in R, 1 \leq i \leq 5 \biggr\}.$
			\item[(iv)] $T_{5,n}=\biggl\{\displaystyle \sum_{j=1}^{q}\left(a_{j} b_{j}-b_{j} a_{j}\right)+\displaystyle \sum_{j=1}^{l} c_{j}^{5}+5 r \mid a_{j}, b_{j}, c_{j}, r \in R, q \geq 1, l \geq 1\biggr\}$.
			\item[(v)] $T_{5,n}=\biggl\{\displaystyle \sum_{j=1}^{q}\left(a_{j} b_{j}-b_{j} a_{j}\right)+c^{5}+5 r \mid a_{j}, b_{j}, c, r \in R, q \geq 1, l \geq 1\biggr\}$.
		\end{enumerate}
	\end{thm} 
	\begin{proof}
		\begin{enumerate}
			\item[(i)] This follows from Lemma \ref{lcalc53}. Since $C(a,b,c,d,e) \in T_{5,n}$ for any $a,b,c,d,e \in R$. Clearly, $a^5 = \tr([\operatorname{diag}(a,0,0,0,0)]^5)$.
			\item[(ii)] By (i), R.H.S. of (ii) $\subset T_{5,3}$. From Theorem \ref{strthm}, $ T_{5,n} \subset$ R.H.S. of (ii).
			\item[(iii)]This follows from Lemma \ref{lbi5} and (ii).
			\item[(iv)] This is Proposition \ref{p1} applied to the expression in (iii).
			\item[(v)] This follows from (iv) after combining several fifth powers into a fifth power and elements from $5R$ using Lemma \ref{lbi5}.
		\end{enumerate}
	\end{proof}
	To conclude, we state:
	\begin{thm}\label{t53final} Let $R$ be a non-commutative ring with unity, let $n = 3$ or $4$. Then, $A \in M_n(R)$ is a sum of fifth powers over $M_n(R)$ if and only if $\tr(A)$ is a sum of a fifth power and some specific five-cycles.
	\end{thm}
	\begin{proof}
		Proof follows similarly to the proof of Theroem \ref{t52final}. By Theorem \ref{t2}, the converse follows.
	\end{proof}
	
	\section{The prime $p = 7$.}
	We now consider the case where the prime $p = 7$. We try to obtain conditions for matrices to be sums of powers over $M_n(R)$ where $n = 4, 5, 6$. Similar results for matrices over commutative rings obtained by Garge \cite{G} are used to obtain the results for matrices over non-commutative rings.
		
	\begin{lem}\label{l7}{\rm [\cite{G}, Lemma 2.8]}
		Let $R$ be a ring. Then every element of the form $7r$, $r \in R$ occurs as a sum of traces of seventh powers over $M_2(R)$.
	\end{lem}
	It should be noted that in the proof, the author uses matrices with entries $1, 0, r, (1+r)$ which commute with each other, being polynomials in $\mathbb{Z}[r]$. Thus, we can use this result in our cases.

		\subsection{The order of the matrices $n = 4, 5, 6,$ and prime power $p = 7$}
		We begin by showing a calculation for the case of $n = 3$.
		\begin{lem}\label{lcalc73}
			Some interesting cycles are observed to belong to $T_{7,3}$ which are later helpful for the case of $T_{7,4}$.
			
			$C(b,g,e,e,e,f,c), C(b,c,b,g,e,f,c), C(a,d,e,a,d,f,g) \in T_{7,3}$.
		\end{lem}
		\begin{proof}
			Consider
			$M_1=
			\begin{bmatrix}
				e & a & f \\
				0 & 0 & b \\
				g & c & 0 \\
			\end{bmatrix}$
			$M_2 = 
			\begin{bmatrix}
				e & -a & f \\
				0 & 0 & b \\
				g & c & 0 \\
			\end{bmatrix}
			$
			$M_3 = 
			\begin{bmatrix}
				e & a & -f \\
				0 & 0 & b \\
				g & c & 0 \\
			\end{bmatrix}$
			
			$M_4 = 
			\begin{bmatrix}
				e & a & f \\
				0 & 0 & b \\
				-g & c & 0 \\
			\end{bmatrix}$
			
			Then $\tr M_1^7 + \tr M_2^7 -\tr M_3^7 - \tr M_4^7 = 
			4(C(e,e,e,e,e,f,g)+ C(b,g,e,e,e,f,c)+ C(b,c,b,g,e,f,c)+			C(e,f,g,f,g,f,g))$ as sums of traces of matrices in $M_4(R)$.
			
			If we replace $b$ by $-b$, in the matrices above, we obtain $4(C(e,e,e,e,e,f,g)- C(b,g,e,e,e,f,c)+ C(b,c,b,g,e,f,c)+	C(e,f,g,f,g,f,g))$. Subtracting these two, we obtain $8C(b,g,e,e,e,f,c) \in T_{7,3}$. Using Lemma \ref{subgrouplem}, and Lemma \ref{l7}, $C(b,g,e,e,e,f,c) \in T_{7,n}$ for $3\leq n \leq 5$. 
		\end{proof}
			
		\begin{lem}\label{lbi7}
			Sums of seventh powers can be represented as a sum of a seventh power and some seven-cyclic sums.
		\end{lem}
		\begin{proof}
			We prove the claim for $(a+b)^7$ where $a, b\in R$, and the remaining proof follows by induction, similar to Lemma \ref{lbi5}.
			
			Then $(a+b)^7 = a^7 + b^7 +
			C(a,a,a,a,a,a,b)+
			C(a,a,a,a,a,b,b)+
			C(a,a,a,a,b,a,b)+
			C(a,a,a,a,b,b,b)+
			C(a,a,a,b,a,a,b)+
			C(a,a,a,b,a,b,b)+
			C(a,a,a,b,b,a,b)+
			C(a,a,a,b,b,b,b)+
			C(a,a,b,a,a,b,b)+
			C(a,a,b,a,b,a,b)+
			C(a,a,b,a,b,b,b)+
			C(a,a,b,b,a,b,b)+
			C(a,a,b,b,b,a,b)+
			C(a,a,b,b,b,b,b)+
			C(a,b,a,b,a,b,b)+
			C(a,b,a,b,b,b,b)+
			C(a,b,b,a,b,b,b)+
			C(a,b,b,b,b,b,b)$.\\
			$= a^7 + b^7 + \displaystyle\sum_{\left(a_{1}, a_{2}, \ldots, a_{7}\right) \in S^\prime} C(a_1, a_2, a_3, a_4, a_5, a_6, a_7)$\\
			
			where $S^\prime \subseteq R^7$ is the set of representatives of $7$-cycles
			shown here, with at least two distinct elements.
					
			So, $a^7 + b^7  = (a+b)^7 - \sum_{\left(a_{1}, a_{2}, \ldots, a_{7}\right) \in S^\prime} C(a_1, a_2, a_3, a_4, a_5, a_6, a_7) = (a+b)^7 + \sum_{\left(a_{1}, a_{2}, \ldots, a_{7}\right) \in S^\prime} C(-a_1, a_2, a_3, a_4, a_5, a_6, a_7).$
						
		\end{proof}
		Another calculation which helps us establish that any $7$-cycle can be expressed as sums of seventh powers of matrices over $M_n(R)$, for $4 \leq n \leq 7$.
		\begin{thm}\label{t74} Let $R$ be a non-commutative ring with unity. Then, the cyclic sum $C(a,b,c,d,e,f,g) \in T_{7,4}$ for any $a, b, c, d, e, f, g \in R$.
		\end{thm}
		\begin{proof}

			Consider 
			$M_1 = \begin{bmatrix}
				0 & a & 0 & 0 \\
				g & 0 & -b & 0 \\
				0 & f & 0 & c \\
				0 & 0 & e & d \\
			\end{bmatrix}$, 
			$M_2 = \begin{bmatrix}
				0 & a & 0 & 0 \\
				g & 0 & b & 0 \\
				0 & f & 0 & c \\
				0 & 0 & e & d \\
			\end{bmatrix}$,
			
			$M_3 = \begin{bmatrix}
				0 & a & 0 & 0 \\
				g & 0 & -b & 0 \\
				0 & f & 0 & c \\
				0 & 0 & -e & d \\
			\end{bmatrix}$,	$
			M_4 = \begin{bmatrix}
				0 & a & 0 & 0 \\
				g & 0 & b & 0 \\
				0 & f & 0 & c \\
				0 & 0 & -e & d \\
			\end{bmatrix}$.
			
			Then $\tr(M_1^7 + M_2^7 - M_3^7 - M_4^7) = 4(C(b,c,d,d,d,e,f)+
			C(a,b,c,d,e,f,g))$. From the case of $3\times 3$ matrices, $C(b,c,d,d,d,e,f)\in T_{7,3}$ by Lemma \ref{subgrouplem}, Lemma \ref{lcalc73} and Corollary \ref{coprimecor}. Thus, $C(a,b,c,d,e,f,g) \in T_{7,n}$ for $4 \leq n \leq 6$.
			
		\end{proof}
		\begin{thm}\label{tcalc74}
			Let $R$ be a non-commutative ring with unity. Let $T_{7,n}$ be the set of those elements of $R$ that can be expressed as sums of traces of seventh powers $M_n(R)$ for $4 \leq n \leq 6$. 
			\begin{enumerate}
				\item[(i)] For $a,b,c,d,e,f,g \in R$, all of the cyclic sums $C(a,b,c,d,e,f,g) \in T_{7,n}$. Also $a^7 \in T_{7,n}$.
				\item[(ii)] $T_{7,n}=\biggl\{\displaystyle \sum_{\left(a_{1}, a_{2}, \ldots, a_{7}\right) \in R^7} C\left(a_{1}, a_{2}, \ldots, a_{7}\right)+\sum_{j=1}^{l} c_{j}^{7} \mid l \geq 1, a_{i}, c_{j} \in R\biggr\} .$
				\item[(iii)] $T_{7,n}=\biggl\{\displaystyle \sum_{\left(a_{1}, a_{2}, \ldots, a_{7}\right) \in R^7} C\left(a_{1}, a_{2}, \ldots, a_{7}\right)+c^{7} \mid a_{i}, c \in R \biggr\}$	
				\item[(iv)] $T_{7,n}=\biggl\{\displaystyle\sum_{j=1}^{q}\left(a_{j} b_{j}-b_{j} a_{j}\right)+\sum_{j=1}^{l} c_{j}^{7}+7 r \mid a_{j}, b_{j}, c_{j}, r\in R, q \geq 1, l \geq 1\biggr\}.$
				\item[(v)] $T_{7,n}=\biggl\{\displaystyle\sum_{j=1}^{q}\left(a_{j} b_{j}-b_{j} a_{j}\right)+c^{7}+7 r \mid a_{j}, b_{j}, c, r \in R, q \geq 1, l \geq 1\biggr\}$.
			\end{enumerate}
		\end{thm} 
		\begin{proof}
			Proof follows similar to the proof of Theorem \ref{t53}, using Theorem \ref{t74}.
		\end{proof}
		
		To conclude, we state:
		\begin{thm} Let $R$ be a non-commutative ring with unity. $A \in M_n(R)$ is a sum of seventh powers $M_n(R)$ if and only if $\tr(A)$ is a sum of a seventh power and some specific seven-cycles for $4 \leq n \leq 6$.
		\end{thm}
		\begin{proof}
			The one-way implication follows from Lemma 5. The proof of the converse follows similar to the proof of the Theorem \ref{t52}, using Lemma \ref{lbi7} .
		\end{proof}

	\begin{rem}
	Let $R$ be a non-commutative ring with unity. As it is found that all the $7$-cycles that appear in the trace of an arbitrary matrix of order $2$ and $3$ were not obtained individually as sums of traces of over $M_3(R)$, the case remains open for $n = 2, 3, p = 7$. However, this was obtained for the case $n =2, p =5$, yet, the `most general' $5$-cycle was not found to be a member of $T_{5,2}$. We conjecture that there may be a lower bound, $q$, such that any arbitrary $p$-cycle can be obtained as sums of traces of $p$-th powers over $M_n(R)$, only for $q \leq n < p$.
	\end{rem}
	
	\section*{Disclosure statement}
	No potential conflict of interest was reported by the author(s). The authors report generative AI was not used in their research or preparation of this manuscript.
	
	\section*{Acknowledgements}
	The authors would like to thank the creators of the software Wolfram-Mathematica, its NCAlgebra\cite{NCA} package, and the Python language. The first author thanks Mrs. Ashwini Kocharekar for the help with programming.
	
	\section*{ORCID}
	Tanay Kochrekar \orcidlink{0009-0003-2838-3963} http://orcid.org/0009-0003-2838-3963
	
	A. S. Garge \orcidlink{0000-0001-9276-3985}
	http://orcid.org/0000-0001-9276-3985
	
	\section*{Author Contributions}
	CRediT: \textbf{Tanay Kochrekar}: Conceptualization, Formal analysis, Investigation,	Methodology, Validation, Writing-original draft, review	\& editing. \textbf{A. S. Garge}: Conceptualization, Formal analysis, Investigation,	Methodology, Supervision, Validation, Writing-review \& editing.
	
\bibliography{waring-noncomm-bib}

\end{document}